\documentclass{amsart}

\usepackage{graphicx}
\usepackage[utf8]{inputenc}
\usepackage{nicefrac}
\usepackage{mathabx}
 \usepackage{subcaption}
\usepackage{amsrefs}
\usepackage{esint}
\usepackage{mathtools}
\mathtoolsset{showonlyrefs}
\usepackage{wrapfig}
\usepackage{xcolor}

\newtheorem{theorem}{Theorem}[section]
\newtheorem{lemma}[theorem]{Lemma}

\theoremstyle{definition}
\newtheorem{definition}[theorem]{Definition}

\newtheorem{cor}[theorem]{Corollary}

\allowdisplaybreaks
\theoremstyle{remark}
\newtheorem{remark}[theorem]{Remark}

\numberwithin{equation}{section}

\allowdisplaybreaks

\newcommand{\mres}{\mathbin{\vrule height 1.6ex depth 0pt width
0.13ex\vrule height 0.13ex depth 0pt width 1.3ex}}

\usepackage{amsmath}
\usepackage{braket} 

\makeatletter
\@namedef{subjclassname@2020}{%
  \textup{2020} Mathematics Subject Classification}
\makeatother

\begin{document}

\title[Polyharmonic equations involving surface measures]{Polyharmonic equations involving \\  surface measures}

\author{Marius Müller}
\address{Albert-Ludwigs-Unversität Freiburg, Mathematisches Institut, 79104 Freiburg im Breisgau}
\email{marius.mueller@math.uni-freiburg.de}

%


\date{\today}


\keywords{Polyharmonic equations; PDEs with measures; signed distance function;  Alt-Caffarelli Problem. \\
\textit{Acknowledgements.} The author would like to thank the anonymous referee for helpful suggestions.}

\begin{abstract}
This article studies (optimal) $W^{2m-1,\infty}$-regularity for the polyharmonic equation $(-\Delta)^m u = Q \; \mathcal{H}^{n-1} \mres \Gamma$, where $\Gamma$ is a (suitably regular) $(n-1)$-dimensional submanifold of $\mathbb{R}^n$, $\mathcal{H}^{n-1}$ is the Hausdorff measure, and $Q$ is some suitably regular density. 
We extend findings in \cite{PoissonMarius}, where the second-order equation $-\mathrm{div}(A(x)\nabla u) = Q \; \mathcal{H}^{n-1} \mres \Gamma$ is studied. As an application we derive (optimal) $W^{3,\infty}$-regularity for solutions of the biharmonic Alt-Caffarelli problem in two dimensions. \\
\end{abstract}
\maketitle


\section{Introduction} \label{section:intro}
Let $\Omega \subset \mathbb{R}^n$ be a $C^\infty$-smooth domain, $n \geq 2$. In this article we prove (optimal) $W^{3,\infty}$-regularity for solutions of the higher order measure-valued equation
\begin{equation}
    (-\Delta)^2 u  = Q \; \mathcal{H}^{n-1} \mres \Gamma \quad \textrm{in $\Omega$}, 
\end{equation}
where $\Gamma \subset \subset \Omega$ is some closed $C^{1,\alpha}$-submanifold and $Q \in C^{0,\alpha}(\Gamma)$ for some $\alpha > 0$. Later we will also study what happens if one replaces  $(-\Delta)^2$ by $(-\Delta)^m$ for some $m \in \mathbb{N}, m \geq 2$. 

For $(-\Delta)^2$ we impose \emph{Navier boundary conditions} with smooth boundary data $u_0 \in C^{\infty}(\overline{\Omega})$, i.e. one would classically require that  $u = u_0$ on $\partial \Omega$ and $\Delta u = 0$ on $\partial \Omega$. The weak formulation of this problem is given by the following
\begin{definition}
We say that $u \in W^{2,2}(\Omega)$ is a weak solution of 
\begin{equation}\label{eq:biharmeqmeas}
    \begin{cases}
    (-\Delta)^2 u = Q \;  \mathcal{H}^{n-1} \mres \Gamma & \textrm{in} \; \Omega \\ u = u_0 , \; \Delta u = 0 & \textrm{on} \; \partial \Omega
    \end{cases}
\end{equation}
if $u- u_0 \in W_0^{1,2}(\Omega)$ and
\begin{equation}\label{eq:Navier}
    \int_\Omega \Delta u \Delta \phi \; \mathrm{d}x = \int_{\Gamma} Q \phi \; \mathrm{d}\mathcal{H}^{n-1} \quad \forall \phi \in C^2(\overline{\Omega})\cap W_0^{1,2}(\Omega).
\end{equation}
\end{definition}
One might notice that the demanded a priori regularity is not enough to speak of classical Navier boundary conditions. The boundary values are however \emph{encoded} in the usage of the large test function space $C^2(\overline{\Omega}) \cap W_0^{1,2}(\Omega)$. Indeed,  \eqref{eq:Navier} is equivalent to $v=- \Delta u$ being a very weak solution of the \emph{measure-valued Dirichlet problem} 
\begin{equation}\label{eq:Poiprob}
    \begin{cases}
   - \Delta v = Q \;  \mathcal{H}^{n-1} \mres \Gamma & \mathrm{in } \;  \Omega \\ 
   \quad  v = 0 & \mathrm{on } \; \partial \Omega,
    \end{cases}
\end{equation}
 in the sense of \cite[Definition 3.1]{Ponce}. 
  If one considers \eqref{eq:Poiprob} for some suitably regular data $Q$, $\Gamma$, one can show that $v = \Delta u \in W^{1,\infty}(\Omega) \cap C_0(\overline{\Omega})$ and the Navier boundary conditions `$\Delta u = 0$ on $\partial \Omega$' also hold classically. We refer to Section \ref{sec:modelprob} for details.
  
  Optimal regularity of weak solutions of \eqref{eq:Poiprob} was also the main object of study in \cite{PoissonMarius}, even for more general elliptic operators than (-$\Delta$). 
  The treatment in Section \ref{sec:modelprob} only resembles the same regularity results as in \cite{PoissonMarius}, 
  but with a more elementary and self-contained argument in the given special case.

Once this regularity for $\Delta u$ is obtained, elliptic regularity implies $u \in W^{3,q}(\Omega)$ for all $q \in [1,\infty)$. 

\vspace{1em}

\textbf{Question 1.} Optimal regularity. 

\vspace{1em}

While $W^{3,q}$-regularity is now understood for each $q < \infty$ it is not clear whether $u \in W^{3,\infty}(\Omega)$. This shall be the main question of this article.

We remark that one can not expect regularity beyond $W^{3,\infty}(\Omega)$ --- e.g. we shall see in Section \ref{sec:modelprob} that $u \not \in C^3(\Omega)$ unless $Q \equiv 0$ (cf. Remark \ref{rem:nonttoC3}).  
Also $W^{4,1}$-regularity is excluded by \eqref{eq:biharmeqmeas}, since $Q \; \mathcal{H}^{n-1}\mres \Gamma \not \in L^1(\Omega)$ (and hence $\Delta^2 u \not \in L^1(\Omega)$). 

This article gives a positive answer to the $W^{3,\infty}$-regularity under some further regularity assumptions on $\Gamma$ and $Q$. 

\begin{theorem}[Optimal regularity]\label{thm:optireg}
Let $Q \in W^{2,p}(\Omega)$ $(p>n)$, $\Gamma = \partial \Omega'$ for some $C^2$-smooth domain $\Omega' \subset \subset \Omega$ 
and $u_0 \in C^\infty(\overline{\Omega})$. Then each weak solution of \eqref{eq:biharmeqmeas} lies in 
$W^{3,\infty}(\Omega)$.
\end{theorem}

From now on we will abbreviate the above condition on $\Gamma$ by `$\Gamma = \partial \Omega' \in C^2$'. We also remark that $Q$ was a priori only defined on $\Gamma$ --- so the condition `$Q \in W^{2,p}(\Omega)$' requires that $Q$ can be extended to a $W^{2,p}$-function on $\Omega$.

One might expect that similar methods as in \cite{PoissonMarius} carry over to the fourth order setting. However, the approach in \cite{PoissonMarius} relies heavily on the \emph{maximum principle}, which is not available for fourth order equations.

Instead our approach relies on a suitable auxiliary equation for $\partial^2_{ij} u$, which we will derive.


\vspace{1em}

\textbf{Question 2.} Inherited regularity for $D^4u$.

\vspace{1em}

Once $W^{3,\infty}$-regularity is shown, one can also ask the question of \emph{regularity inheritance}. This is a different question from the above, as we shall explain in the following. Instead of searching the highest Sobolev class that solutions live in, we ask whether all highest order derivatives inherit their quality from the operator. For instance, for $p \in (1,\infty)$ it is known that (-$\Delta$) inherits $L^p$-regularity to $D^2u$ in the following sense: Looking at all derivatives in the sense of distributions, one has
\begin{equation}
-\Delta u \in L^p_{loc} \quad \Rightarrow  \quad  D^2 u \in L^p_{loc}
\end{equation}
for each $u \in L^1_{loc}( \mathbb{R}^n)$. 
Let now $\mathcal{M}_{loc}(\Omega)$ denote the space of all signed Radon measures on $\Omega$. It seems interesting to study under which conditions $-\Delta u \in \mathcal{M}_{loc}(\Omega)$ implies that $D^2u \in \mathcal{M}_{loc}(\Omega)$.
\begin{definition}[Inheritance of $\mathcal{M}_{loc}$-regularity for (-$\Delta$)]\label{def:13}
Let $S \subset \mathcal{M}_{loc}(\Omega)$ be a subspace. We say that the differential inclusion `$-\Delta u \in S$' inherits $\mathcal{M}_{loc}$-regularity to $D^2u$ if for each function $u \in L^1_{loc}(\Omega)$ one has (in the sense of distributions)
\begin{equation}
-\Delta u \in S \quad \Rightarrow \quad D^2u \in  \mathcal{M}_{loc}(\Omega). \label{eq:maxmloc-reg}
\end{equation}
\end{definition}
\begin{remark}
The conclusion $D^2u \in \mathcal{M}_{loc}(\Omega)$ can be phrased equivalently via $\nabla u \in BV_{loc}(\Omega)$. This will be useful for our proofs.   
\end{remark}

\begin{remark}
Notice that we have to restrict to a subspace $S \subsetneq \mathcal{M}_{loc}(\Omega)$  in Definition \ref{def:13} since conclusion \eqref{eq:maxmloc-reg} is not true for $S= \mathcal{M}_{loc}(\Omega)$. This becomes visible if one looks at $\Omega=  \mathbb{R}^n, S = \mathcal{M}_{loc}(\mathbb{R}^n)$ and considers $\mu= \delta_0 \in S$. The fundamental solution $F= F_{\Delta} \in L^1_{loc}(\mathbb{R}^n)$ solves then (distributionally) $-\Delta F = \delta_0 \in S$ but one has $D^2 F \not \in \mathcal{M}_{loc}(\mathbb{R}^n)$. The latter can be easily seen as follows. If $D^2F \in \mathcal{M}_{loc}(\mathbb{R}^n)$ then $\nabla F \in BV_{loc}(\mathbb{R}^n)$, which would imply $\nabla F \in L^{ \frac{n}{n-1} }_{loc}( \mathbb{R}^n)$. Noticing that $\nabla F(x) = \frac{x}{|x|^n}$ a.e. for any $n \geq 2$ and one readily checks  that $\nabla F \not \in L^{ \frac{n}{n-1} }_{loc}( \mathbb{R}^n)$.
\end{remark}
 

In this article the studied subspace $S$ of Definition \ref{def:13} is 
\begin{equation}\label{eq:Subspacemax}
 S := S(\Gamma,p) := \{ Q \; \mathcal{H}^{n-1} \mres \Gamma : Q \in W^{2,p}(\Omega) \}
\end{equation}
for some $p > n$ and some $\Gamma= \partial \Omega' \in C^2$. We will see in Section \ref{sec:modelprob} that $-\Delta u \in S(\Gamma,p)$ inherits $\mathcal{M}_{loc}$-regularity. This serves as a starting point for the investigation of regularity inheritance properties for (-$\Delta$)$^2$. Since the operator is of fourth order, it is now natural to ask for regularity inheritance to $D^4u$. More precisely, we ask whether $(-\Delta)^2 u \in S(\Gamma, p)$ implies that $D^4 u \in \mathcal{M}_{loc}(\Omega)$. The statement `$D^4 u \in \mathcal{M}_{loc}(\Omega)$' can again best be phrased by $D^3 u \in BV_{loc}(\Omega)$. 
 This $BV$-regularity turns out to be true if $\Gamma$ is additionally assumed $C^{2,1}$-smooth. 

\begin{theorem}[Regularity inheritance]\label{thm:bvreg}
Let $Q \in W^{2,p}(\Omega)$ $(p>n)$, $\Gamma = \partial \Omega' \in C^{2,1}$ and $u_0 \in C^\infty(\overline{\Omega})$. Then each weak solution \eqref{eq:biharmeqmeas} satisfies $D^3 u \in BV(\Omega)$.
\end{theorem}

Once $BV$-regularity of $D^3u$ is shown it is natural to ask the more specific question of $SBV$-regularity of $D^3u$, which we will answer positively under the same conditions. In doing so, we also have to refine our study for the model equation \eqref{eq:Poiprob}.


\vspace{1em}

\textbf{Application.} The biharmonic Alt-Caffarelli problem.

\vspace{1em}

The results can be applied to prove optimal regularity for minimizers of the \emph{biharmonic Alt-Caffarelli Problem} (in $2$d), which has recently raised a lot of interest, cf. \cite{Serena1}, \cite{Serena2}, \cite{AltCaffarelliMarius},  \cite[Section 5]{PoissonMarius}. Subject of this is the minimization of the functional
\begin{equation}
    \mathcal{E}(u) := \int_\Omega (\Delta u)^2  \; \mathrm{d}x + | \{ u > 0 \} | 
\end{equation}
among all $u \in W^{2,2}(\Omega)$ such that $u- u_0 \in W_0^{1,2}(\Omega)$ for some $u_0 \in C^\infty(\overline{\Omega}),u_0> 0$.  
We will discuss the details in Section \ref{sec:AltCaffarelli}.


\section{The model equation $-\Delta u = Q  \; \mathcal{H}^{n-1} \mres \Gamma$}
In what follows we always assume (unless stated otherwise) that $\Omega \subset \mathbb{R}^n$ is a $C^\infty$-smooth domain and  $\Gamma = \partial \Omega'\in C^2$. We call $\Omega'' := \Omega \setminus \overline{\Omega'}$ throughout this section. Moreover, if not stated otherwise $Q \in C^{0,\alpha}(\Gamma)$ is a Hölder continuous function on $\Gamma$ for some $\alpha > 0$. 

\subsection{Preliminaries}\label{sec:prelom}
Before we can study the fourth order problem we need to refine some results for the second order problem \eqref{eq:Poiprob}.  
\begin{definition}
We say that $v \in L^2(\Omega)$ is a (very) weak solution of \eqref{eq:Poiprob} if 
\begin{equation}\label{eq:weaksecsol}
-\int_\Omega v \Delta \phi \; \mathrm{d}x = \int_\Gamma Q \phi \; \mathrm{d}\mathcal{H}^{n-1} \quad \forall \phi \in C^2(\overline{\Omega}) \cap W_0^{1,2}(\Omega). 
\end{equation}
\end{definition}

\begin{remark}\label{rem:LaxMilgram}
Observe that the map $T : W_0^{1,2}(\Omega) \rightarrow \mathbb{R}$ given by $T(\phi) := \int_\Gamma Q \phi \; \mathrm{d} \mathcal{H}^{n-1}$ defines an element of the dual space $W_0^{1,2}(\Omega)^*$ (due to the continuity of the Sobolev trace operator). Hence, existence of solutions of \eqref{eq:weaksecsol} is readily checked with the Lax-Milgram theorem. Solutions constructed this way actually lie in $W_0^{1,2}(\Omega)$. Even more regularity can be obtained with this method: Properties of the Sobolev trace yield also that $T \in W_0^{1,s}(\Omega)^*$ for all $s \in (1,\infty)$. As $\Delta : W_0^{1,q}(\Omega) \rightarrow W_0^{1,\frac{q}{q-1}}(\Omega)^*$ is an isomorphism for all $q \in (1,\infty)$ we obtain that solutions lie in $ W_0^{1,q}(\Omega)$ for all $q< \infty$. 
\end{remark}

Most of the results in this section can also be obtained from \cite{PoissonMarius}, where the equation $-\mathrm{div}(A(x) \nabla u) = Q \; \mathcal{H}^{n-1} \mres \Gamma$ for an elliptic operator $A \in W^{1,q}(\Omega)$ and a density $Q \in C^{0,\alpha}(\Gamma)$ on a surface $\Gamma= \partial\Omega' \in C^{1,\alpha}$ was studied.
Our approach is self-contained in the sense that results from \cite{PoissonMarius} are not used but rather reproved in a more elementary way for the special case of $A = I_n, Q \in W^{2,p}(\Omega), (p>n)$ and $\Gamma= \partial \Omega' \in C^2$.
The key ingredient is a comparison with the \emph{signed distance function}
\begin{equation}
    d_\Gamma : B_\epsilon(\Gamma) \rightarrow \mathbb{R} \quad d_\Gamma(x) := \begin{cases}
    -\mathrm{dist}(x,\Gamma) & x \in \Omega' \\ 0 & x \in \Gamma \\ \mathrm{dist}(x,\Gamma) & x \in \Omega'', 
    \end{cases}
\end{equation}
where $B_\epsilon(\Gamma) := \{ x \in \mathbb{R}^n : \mathrm{dist}(x,\Gamma) < \epsilon \}$ for some  $\epsilon > 0$, which we will always fix and choose small enough as in \cite[Appendix 14.6]{GilTru} (and in such a way that $B_\epsilon(\Gamma) \subset \subset \Omega$). We briefly recall some properties of $d_\Gamma$ on $B_\epsilon(\Gamma)$ from \cite[Appendix 14.6]{GilTru}. First note that $B_\epsilon(\Gamma)$ is a $C^2$-domain and $d_\Gamma \in C^2(\overline{B_\epsilon(\Gamma)})$. Further, one has $\nabla d_\Gamma = \nu^{\Omega'} \circ \pi_\Gamma$ where $\pi_\Gamma : B_\epsilon(\Gamma) \rightarrow \Gamma$ is the ($C^1$-regular) \emph{nearest point projection} on $\Gamma$. 
We will always denote $\nu := \nu^{\Omega'}$ from now on.

\begin{lemma}\label{lem:distprep}
Let $Q \in W^{2,p}(\Omega)$, $p>n$ and let $\Gamma  = \partial \Omega' \in C^2$ with outer unit normal $\nu$. Further let $d_\Gamma : B_\epsilon(\Gamma) \rightarrow \mathbb{R}$ be the signed distance function of $\Gamma$. Then (distributionally in $C_0^\infty(B_\epsilon(\Gamma))'$) there holds 
\begin{equation}
    \partial^2_{ij}\left( \frac{Q}{2}|d_\Gamma| \right)  = Q \nu_i \nu_j \mathcal{H}^{n-1}\mres \Gamma + g_{ij}
\end{equation}
for some $g = (g_{ij}) \in L^p(B_\epsilon(\Gamma); \mathbb{R}^{n\times n})$. 
In particular,
\begin{equation}\label{eq:Laplace}
   - \Delta \left( \frac{Q}{2}|d_\Gamma| \right) = - Q \; \mathcal{H}^{n-1} \mres \Gamma - \mathrm{tr}(g).
\end{equation}
\end{lemma}
\begin{proof}
Let 
$\phi \in C_0^\infty(B_\epsilon(\Gamma))$ be arbitrary but fixed. Denote by $\nu = \nu^{\Omega'}$ and $\nu^{\Omega''}$ the outer unit normals of $\Omega', \Omega''$ (respectively).
Then using integration by parts and $d_\Gamma \vert_\Gamma = 0$ we find
\begin{align}
    \int \frac{Q}{2}|d_\Gamma| \partial^2_{ij} \phi \; \mathrm{d}x & = -\int_{\Omega'} \frac{Q}{2} d_\Gamma \partial^2_{ij} \phi \; \mathrm{d}x + \int_{\Omega''} \frac{Q}{2} d_\Gamma \partial^2_{ij} \phi \; \mathrm{d}x
    \\ & = - \int_{\partial \Omega'} \frac{Q}{2}d_\Gamma \nu_i^{\Omega'} \partial_j \phi \; \mathrm{d}\mathcal{H}^{n-1} + \int_{\partial \Omega''} \frac{Q}{2}d_\Gamma \nu_i^{\Omega''} \partial_j \phi  \; \mathrm{d}\mathcal{H}^{n-1}\\ 
    & \quad + \int_{\Omega'} \partial_i \left( \frac{Q}{2} d_\Gamma \right) \partial_j \phi \; \mathrm{d}x -  \int_{\Omega''} \partial_i \left( \frac{Q}{2} d_\Gamma \right) \partial_j \phi \; \mathrm{d}x
    \\ & = \int_{\Omega'} \partial_i \left( \frac{Q}{2} d_\Gamma \right) \partial_j \phi \; \mathrm{d}x -  \int_{\Omega''} \partial_i \left( \frac{Q}{2} d_\Gamma \right) \partial_j \phi \; \mathrm{d}x
    \\ & = \int_{\partial \Omega'} \partial_i \left(\frac{Q}{2}d_\Gamma\right) \nu^{\Omega'}_j \phi \; \mathrm{d}\mathcal{H}^{n-1} - \int_{\partial \Omega''} \partial_i \left(\frac{Q}{2}d_\Gamma\right) \nu^{\Omega''}_j \phi \; \mathrm{d}\mathcal{H}^{n-1}\\
    & \quad - \int_{\Omega'} \partial^2_{ij} \left( \frac{Q}{2} d_\Gamma \right)  \phi \; \mathrm{d}x + \int_{\Omega''} \partial^2_{ij} \left( \frac{Q}{2} d_\Gamma \right)  \phi \; \mathrm{d}x.
\end{align}
Noticing that $\partial \Omega' \cap \mathrm{supp}(\phi) = \partial \Omega'' \cap \mathrm{supp}(\phi) = \Gamma$ and $\nu_j^{\Omega''} = - \nu_j^{\Omega'}$ on $\Gamma$ we infer 
\begin{equation}
     \int \frac{Q}{2}|d_\Gamma| \partial^2_{ij} \phi \; \mathrm{d}x  = \int_\Gamma \partial_i (Q d_\Gamma) \nu_j \phi \; \mathrm{d} \mathcal{H}^{n-1} + \int g_{ij} \phi \; \mathrm{d}x,
\end{equation}
where $g_{ij} := \partial^2_{ij}( \frac{Q}{2} d_\Gamma) ( \chi_{\Omega''} - \chi_{\Omega'}) \in L^p(B_\epsilon(\Gamma))$. Now notice that on $\Gamma$ one has 
\begin{equation}
    \partial_i (Q d_\Gamma) = (\partial_i Q) d_\Gamma + Q (\partial_i d_\Gamma) = 0 + Q \nu_i = Q \nu_i.
\end{equation}
Thus we infer 
\begin{equation}
     \int \frac{Q}{2}|d_\Gamma| \partial^2_{ij} \phi \; \mathrm{d}x = \int_\Gamma Q \nu_i \nu_j  \phi \; \mathrm{d}\mathcal{H}^{n-1} + \int g_{ij} \phi \; \mathrm{d}x \qquad \forall \phi \in C_0^\infty(B_\epsilon(\Gamma)),
\end{equation}
which was asserted. Formula \eqref{eq:Laplace} follows immediately from $- \Delta f  = - \sum_{i = 1}^n \partial^2_{ii} f$ and 
and using $\sum_i \nu_i^2 = 1$. 
\end{proof}

\begin{cor}\label{cor:siigndist}
Let $\Gamma = \partial \Omega' \in C^2$ and let $Q \in W^{2,p}(\Omega)$ $(p> n)$, be such that $Q\vert_\Gamma \not \equiv 0$. Then 
\begin{equation}
\textrm{(a)} \; \;  \tfrac{Q}{2}|d_\Gamma| \in W^{1,\infty}(B_\epsilon(\Gamma)), \quad \textrm{(b)} \; \; \nabla( \tfrac{Q}{2}|d_\Gamma|)  \in BV(B_\epsilon(\Gamma)), \quad \textrm{(c)} \; \;  \tfrac{Q}{2}|d_\Gamma| \not \in C^1(B_\epsilon(\Gamma)).
\end{equation}
\end{cor}
\begin{proof}
Claim (a) is immediate by Lipschitz continuity of $|d_\Gamma|$ and the fact that  $W^{2,p} \hookrightarrow C^1$ for $p>n$. Claim (c) follows from the fact that $\nabla |d_\Gamma| = \mathrm{sgn}(d_\Gamma) \nabla d_\Gamma = \mathrm{sgn}(d_\Gamma) \nu \circ \pi_\Gamma$ has a jump at each point on $\Gamma$. 
 To show claim (b) we compute for $i = 1,...,n$  and $\psi \in C_0^\infty(B_\epsilon(\Gamma); \mathbb{R}^n)$ with $|\psi| \leq 1$ using Lemma \ref{lem:distprep}
\begin{align}
& \int \partial_i (\tfrac{Q}{2}d_\Gamma) \mathrm{div}(\psi) \; \mathrm{d}x   = \sum_{j = 1}^n \int \partial_i( \tfrac{Q}{2}d_\Gamma) \partial_j \psi_j \; \mathrm{d}x 
= - \sum_{j = 1}^n \int \tfrac{Q}{2}d_\Gamma \partial^2_{ij} \psi_j \; \mathrm{d}x 
\\ & = - \sum_{j = 1}^n  \left( \int_\Gamma Q \nu_i \nu_j \psi_j \; \mathrm{d}\mathcal{H}^{n-1} -  \int g_{ij} \psi_j \; \mathrm{d}x \right) \leq \sum_{j=1}^n (||Q||_{L^1(\Gamma)} + ||g_{ij}||_{L^1(B_\epsilon(\Gamma))}).
\end{align} 
This implies $\partial_i (\tfrac{Q}{2}d_\Gamma) \in BV(B_\epsilon(\Gamma))$ and thus (b) follows.  
\end{proof}

\subsection{Regularity results in the spirit of \cite{PoissonMarius}}\label{sec:modelprob}

We can now retrieve the results of \cite[Theorem 1.2]{PoissonMarius} as well as \cite[Corollary 3.3 and Remark 3.4]{PoissonMarius} and \cite[Lemma 2.6]{PoissonMarius} (in our special case) in an easier fashion than in \cite{PoissonMarius}. As a byproduct we will obtain a maximal $\mathcal{M}_{loc}$-regularity statement for (-$\Delta$) with $S$ as in \eqref{eq:Subspacemax}. 
\begin{lemma}\label{lem:alles}
Let $v \in L^2(\Omega)$ satisfy \eqref{eq:weaksecsol} for some $Q \in W^{2,p}(\Omega)$ $(p> n)$ with $Q \vert_\Gamma \not \equiv 0$. Then
\begin{equation}
\mathrm{(i)} \; \;  v \in W^{1,\infty}(\Omega), \qquad \mathrm{(ii)} \; \; \nabla v \in BV(\Omega), \qquad \mathrm{(iii)} \; \;  v \not \in C^1(\Omega).
\end{equation}
\end{lemma}
\begin{proof}
Let $\phi \in C_0^\infty(B_\epsilon(\Gamma))$. Then by \eqref{eq:Laplace} we have
\begin{equation}
- \int v \Delta \phi \; \mathrm{d}x = \int_\Gamma Q \phi \; \mathrm{d}\mathcal{H}^{n-1} = \int \frac{Q}{2}|d_\Gamma| \Delta \phi \; \mathrm{d}x - \int_\Omega \mathrm{tr}(g) \phi \; \mathrm{d}x. 
\end{equation}
In particular, we have (distributionally in $C_0^\infty(B_\epsilon(\Gamma))'$) that $(-\Delta) (v + \frac{Q}{2}|d_\Gamma|) = \mathrm{tr}(g) \in L^p(B_\epsilon(\Gamma))$. We infer by elliptic regularity that $v + \frac{Q}{2}|d_\Gamma| \in W^{2,p}_{loc}(B_\epsilon(\Gamma))$. Let now $h \in W^{2,p}_{loc}(B_\epsilon(\Gamma))$ be such that $v = -\frac{Q}{2}|d_\Gamma| + h$. The fact that $h \in W^{2,p}_{loc} \subset C^1$ implies together with Corollary \ref{cor:siigndist} (a) that $v \in W^{1,\infty}_{loc}(B_\epsilon(\Gamma))$.  Moreover, $v \not \in C^1(B_\epsilon(\Gamma))$ is easily deduced by contradiction ---  If $v$ were  an element of $C^1(B_\epsilon(\Gamma))$, then also $v- h \in C^1(B_\epsilon(\Gamma))$. But $v-h =- \frac{Q}{2}|d_\Gamma| \in C^1(B_\epsilon(\Gamma))$ contradicts Corollary \ref{cor:siigndist} (c). This already implies (iii). Finally, $\nabla h\in W_{loc}^{1,p} \subset BV_{loc}$ and Corollary \ref{cor:siigndist} (b) can be used to show $v = - \frac{Q}{2} |d_\Gamma| + h \in BV_{loc}(B_\epsilon(\Gamma))$. Now observe by \eqref{eq:weaksecsol} that $v$ is weakly harmonic on $\Omega \setminus \Gamma$, implying that $v \in C^\infty(\Omega \setminus \Gamma)$. In particular $w := v \vert_{\partial B_{\frac{\epsilon}{2}}(\Gamma)} \in C^2(\partial B_{\frac{\epsilon}{2}}(\Gamma))$. Then $v$ solves on $\Omega \setminus B_{\frac{\epsilon}{2}}(\Gamma)$ (weakly) 
\begin{equation}
\begin{cases}
-\Delta v = 0 & \textrm{in $\Omega \setminus B_{\frac{\epsilon}{2}}(\Gamma)$}, \\ \quad \; \;  v = 0 & \textrm{on $\partial \Omega$}, \\ \quad \; \; v = w & \textrm{on $\partial  B_{\frac{\epsilon}{2}}(\Gamma)$.}
\end{cases}
\end{equation}
Elliptic regularity thus implies that $v \in W^{2,q}(\Omega\setminus B_{\frac{\epsilon}{2}}(\Gamma))$ for any $q< \infty$.  This together with the fact that $v \in W^{1,\infty}(B_{\frac{3\epsilon}{4}}(\Gamma))$ and $ \nabla v \in BV(B_{\frac{3\epsilon}{4}}(\Gamma))$ implies (i) and (ii).
\end{proof}

\begin{remark}\label{rem:nonttoC3}
From point (iii) in the previous lemma and \eqref{eq:Poiprob} it follows that each weak solution $u \in W^{2,2}(\Omega)$ of \eqref{eq:biharmeqmeas}  with $Q \vert_\Gamma \not\equiv 0$ satisfies $v = -\Delta u \not \in C^1(\Omega)$. In particular $u \not \in C^3(\Omega)$ unless $Q\vert_\Gamma \equiv 0$. 
\end{remark}

As a direct consequence of point (ii) in the previous lemma we obtain the regularity inheritance for the second order problem.

\begin{cor}
Let $S= S(\Gamma,p)$ be as in \eqref{eq:Subspacemax} with $p > n$. Then `$-\Delta u \in S$' inherits $\mathcal{M}_{loc}$-regularity with respect to $S$ (in the sense of Definition \ref{def:13}). 
\end{cor}
\begin{proof}
Let $w \in L^1_{loc}(\Omega)$ be such that (distributionally) $-\Delta w \in S$, i.e. there exists $Q \in W^{2,p}(\Omega)$ such that $-\Delta w = Q \; \mathcal{H}^{n-1} \mres \Gamma$. By Remark \ref{rem:LaxMilgram} we have the existence of some $v \in W_0^{1,2}(\Omega)\subset L^2(\Omega)$ weak solution of 
\begin{equation}
\begin{cases} 
-\Delta v = Q \; \mathcal{H}^{n-1} \mres \Gamma  & \textrm{in $\Omega$} \\ \quad \; \; v = 0 & \textrm{on $\partial \Omega$}
\end{cases}.
\end{equation}
Lemma \ref{lem:alles} now implies that $\nabla v \in BV(\Omega)$. Since $w-v \in L^1_{loc}(\Omega)$ is weakly harmonic in $\Omega$ we also deduce that $\nabla (w-v) \in BV_{loc}(\Omega)$. These observations imply $\nabla w \in BV_{loc}(\Omega)$ and therefore $D^2 v \in \mathcal{M}_{loc}(\Omega)$. The claim follows.  
\end{proof}

We will also need a lemma about (local) Lipschitz-regularity of distributional solutions to $-\Delta w = Q \; \mathcal{H}^{n-1} \mres \Gamma$.
\begin{lemma}\label{lem:LocLip}
Let $D \subset \mathbb{R}^n$ be any $C^2$-domain, $\Gamma = \partial D' \in C^2$ for some $D' \subset \subset D$ and $w \in L^1_{loc}(D)$ satisfy (distributionally in $C_0^\infty(D)'$) 
\begin{equation}
   - \Delta w = Q \; \mathcal{H}^{n-1} \mres \Gamma, 
\end{equation}
for some $Q \in W^{2,p}(D)$, $(p>n)$. Then $w \in W^{1,\infty}_{loc}(D).$
\end{lemma}
\begin{proof}
For $\delta > 0$ consider $D_\delta := \{ x \in D : \mathrm{dist}(x,D^C) > \delta\}$. By \cite[Appendix 14.6]{GilTru} $D_\delta$ is a $C^2$-domain for all $\delta \in (0, \delta_0)$ small enough. Possibly choosing $\delta_0 > 0$ smaller, one can also achieve that $\Gamma \subset \subset D_{2\delta}$ for all $\delta \in (0,\delta_0)$.  Fix some arbitrary $\delta \in (0, \frac{\delta_0}{2})$. By \cite[Theorem 5.1]{Doktor} there exists some sequence $(U_k)_{k \in \mathbb{N}}$ of domains such that $U_k \supset D_\delta$ has $C^\infty$-smooth boundary and $U_k$ approximates $D_\delta$ in a suitable sense (whose precise definition will not be needed here). With \cite[Theorem 5.1]{Doktor} one also obtains $U_k \subset D$ for some $k \geq k_0$ large enough. Next consider the measure-valued Dirichlet problem 
\begin{equation}\label{eq:wschlaange}
    \begin{cases}
    -\Delta \tilde{w} = Q \; \mathcal{H}^{n-1} \mres \Gamma & \textrm{on} \; U_{k_0} \\ \quad \tilde{w} = 0 & \textrm{on} \; \partial U_{k_0}. 
    \end{cases}
\end{equation}
By Lemma \ref{lem:alles}, \eqref{eq:wschlaange} has a solution $\tilde{w} \in W^{1,\infty}(U_{k_0})$. In particular $\tilde{w}$ is also Lipschitz continuous on $\overline{D_{\delta}}$. Now note that  $\bar{w} := w-\tilde{w}$ satisfies $-\Delta \bar{w} = 0$ distributionally on $D_{\delta}$, meaning that $\bar{w} \in C^\infty(\overline{D_{2\delta}})$. This implies that $w = \tilde{w} + \bar{w} \in W^{1,\infty}(D_{2\delta})$. Since $\delta \in (0,\frac{\delta_0}{2})$ was arbitrary the claim follows. 
\end{proof}

\subsection{SBV-Regularity}

We first briefly report on the definition of the space $SBV$. Recall that for each $BV$-function $w \in BV(\Omega)$ there exists a finite signed vector-valued Radon measure $\mu_w$ on $\Omega$ such that 
\begin{equation}
    \int_\Omega w \; \mathrm{div}(\phi) \; \mathrm{d}x = \int_\Omega \phi \; \mathrm{d}\mu_w \quad \forall \phi \in C^1_c(\Omega;\mathbb{R}^n).
\end{equation}
Now $\mu_u$ can be decomposed (cf. \cite[Section 5.1]{EvGar}) into two finite signed vector-valued measures $\mu_w = \mu_w^a + \mu_w^s$, where $\mu_w^a$ is absolutely continuous with respect to the Lebesgue measure and $\mu_w^s$ is singular with respect to the Lebesgue measure on $\Omega$. 

Further we define for each Lebesgue-measurable $f : \Omega \rightarrow \mathbb{R}$  the \emph{approximate limits}
\begin{equation}
    f_+(x) := \inf \left\lbrace t \in \mathbb{R} : \lim_{r \rightarrow 0 } \frac{|\{ f > t \}\cap B_r(x)|}{r^n} = 0 \right\rbrace 
\end{equation}
and 
\begin{equation}
    f_-(x) := \sup \left\lbrace t \in \mathbb{R} : \lim_{r \rightarrow 0 } \frac{|\{ f < t \}\cap B_r(x)|}{r^n} = 0 \right\rbrace. 
\end{equation}
It can be shown that $f_+,f_-$ are Borel measurable and $f_+ \geq f_-$, cf. \cite[Chapter 5, Lemma 5.6]{EvGar}.
We define the \emph{jump set} of such $f$ to be the Borel set
\begin{equation}
    J_f := \{x \in \Omega : f_+(x) > f_-(x) \}. 
\end{equation}
If $w \in BV(\Omega; \mathbb{R})$ then \cite[Chapter 5, Theorem 5.17]{EvGar} implies that $J_w$ is countably $(n-1)$-rectifiable. Now we are ready to define the space $SBV(\Omega)$.
\begin{definition}[The space $SBV(\Omega)$]
We say $w \in BV(\Omega)$ lies in $SBV(\Omega)$ if $\mu_w^s ( \Omega \setminus J_w) = 0$.
\end{definition}

We will make use the following characterization lemma of $SBV$ which is frequently used in the study of free discontinuity problems

\begin{lemma}[{cf. \cite[Lemma 2.3]{DeGiorgi}}]\label{lem:DeGiorgi}
Let $w \in L^\infty(\Omega)$ and $K \subset \Omega$ be (relatively) closed such that $w \in C^1(\Omega \setminus K) \cap BV(\Omega)$ and $\mathcal{H}^{n-1}(K) < \infty$. Then $w \in SBV(\Omega)$ and $J_w \subset K$. 
\end{lemma}

\begin{cor}\label{cor:33}
Let $\Gamma = \partial \Omega' \in C^2$ and $Q \in W^{2,p}(\Omega)$, $p > n$. Further, let $v \in L^1(\Omega)$ be a weak solution of
\begin{equation}
\begin{cases}
-\Delta v = Q \; \mathcal{H}^{n-1} \mres \Gamma & \textrm{in }\Omega, \\ \quad \; \; v = 0 & \textrm{on } \partial \Omega.
\end{cases}
\end{equation} 
Then $v \in W^{1,\infty}(\Omega)$ and  $\nabla v \in SBV(\Omega)$. Moreover, $J_{\partial_i v} \subset  \Gamma$ for all $i = 1,...,n$.
\end{cor}
\begin{proof}
We intend to apply Lemma \ref{lem:DeGiorgi}.
By Lemma \ref{lem:alles} one has that $v \in W^{1,\infty}(\Omega)$ (i.e. $\nabla v \in L^\infty(\Omega)$) and also from Lamma \ref{lem:alles} we conclude that $\nabla v \in BV(\Omega)$. Clearly $v$ lies also in $C^1(\Omega \setminus \Gamma)$ as $v$ is (weakly) harmonic on $\Omega \setminus \Gamma$. Since $\Gamma \subset \subset \Omega$ is a closed $C^1$-submanifold, we infer also that $\mathcal{H}^{n-1}(\Gamma) < \infty$. Hence, Lemma \ref{lem:DeGiorgi} is applicable and the claim follows. 
\end{proof}

\section{Optimal regularity (Proof of Theorem \ref{thm:optireg})}  We will prove the theorem by showing that for all $u \in W^{2,2}(\Omega)$ as in Theorem \ref{thm:optireg} and $i,j \in \{ 1,...,n \}$ the function $\partial^2_{ij} u \in L^2(\Omega)$ lies actually in $W^{1,\infty}(\Omega)$. The reason for that will be that it solves a suitable measure-valued \emph{auxiliary problem} of second order. Together with our results from the previous section this will yield the desired regularity. 

We have already discussed in \eqref{eq:Poiprob} and thereafter that $\Delta u \in W^{1,\infty}(\Omega) \cap C_0(\overline{\Omega})$ (meaning  also $\Delta u = 0$ on $\partial \Omega$) and $u \in W^{3,q}(\Omega)$ for all $q \in [1,\infty)$. Next we turn to some local smoothness properties on $\Omega \setminus \Gamma$.
\begin{lemma}\label{lem:LocBiharm}
Let $u \in W^{2,2}(\Omega)$ be a weak solution of \eqref{eq:biharmeqmeas}. Then $u \in C^\infty( \Omega \setminus \Gamma)$ and $(-\Delta)^2 u = 0$ on $\Omega \setminus \Gamma$. Moreover, $u \in C^\infty(\overline{\Omega} \cap N)$ for an open neighborhood $N$ of $\partial \Omega$. 
\end{lemma}
\begin{proof}
Notice that for each $\phi \in C_0^\infty(\Omega \setminus \Gamma)$ one has by \eqref{eq:Navier} that 
\begin{equation}
    \int_\Omega \Delta u \Delta \phi \; \mathrm{d}x = \int_\Gamma Q \phi \; \mathrm{d}\mathcal{H}^{n-1} = 0. 
\end{equation}
In particular $\Delta u$ is weakly harmonic on $\Omega \setminus \Gamma$. This implies that $\Delta u$ lies in $C^\infty(\Omega \setminus \Gamma)$ and is harmonic on $\Omega \setminus \Gamma$. Elliptic regularity now implies that $u \in C^\infty(\Omega \setminus \Gamma)$ and harmonicity of $\Delta u$ yields $\Delta^2 u = \Delta(\Delta u) = 0$. Now let $D \subset \subset \Omega$ be a smooth subdomain such that $\Gamma \subset \subset D \subset \subset \Omega$. [Such subdomain clearly exists since one can observe that for $\epsilon > 0$ small enough the subdomain $\Omega_\epsilon := \{ x \in \mathbb{R}^n : \mathrm{dist}(x,\Omega^C) > \epsilon \}$ has $C^2$-boundary and can hence be appoximated by smooth domains with the same construction as in the Lemma \ref{lem:LocLip}]. Define $\varphi := (\Delta u) \vert_{\partial D} \in C^\infty(\partial D)$. 
Notice that $\Omega \setminus D$ is a smooth domain and a subset of $\Omega \setminus \Gamma$. Therefore, $v := \Delta u$ is harmonic on $\Omega \setminus D$. Moreover $v$ is continuous on $\overline{\Omega \setminus D}$ with boundary values $v = 0$ in $\partial \Omega$ and on $\partial D$ one has $v= \varphi \in C^\infty(\partial D)$. In particular $v$ is smooth on $\partial ( \Omega \setminus D) $. Elliptic regularity and smoothness of $\Omega \setminus D$ imply then that $v \in C^\infty( \overline{\Omega \setminus D}).$ The claim follows choosing $N := \mathbb{R}^n \setminus \overline{D}$.
\end{proof}

Finally we are able to find the desired auxiliary equation for $\partial^2_{ij} u$ and prove Theorem \ref{thm:optireg}.

\begin{proof}[Proof of Theorem \ref{thm:optireg}]
Let $u$ be as in the statement and $1 \leq i,j \leq n$. We first derive a distributional equation for $\partial^2_{ij} u$ on $C_0^\infty(B_\epsilon(\Gamma))'$.
To this end let $\phi \in C_0^\infty(B_\epsilon(\Gamma))$ be arbitrary. Then several integrations by parts yield
\begin{align}
    \int \partial^2_{ij} u \Delta \phi \; \mathrm{d}x  & = \int u \Delta (\partial^2_{ij} \phi ) \; \mathrm{d}x = \int \Delta u \partial^2_{ij} \phi \; \mathrm{d}x
    \\ & =  \int  \left( \Delta u - \frac{Q}{2} |d_\Gamma| \right) \partial^2_{ij} \phi \; \mathrm{d}x + \int \left( \frac{Q}{2} |d_\Gamma| \right) \partial^2_{ij}\phi \; \mathrm{d}x. \label{eq:prfintermediate} 
\end{align}
Recall that distributionally in $C_0^\infty(B_\epsilon(\Gamma))'$ one has
\begin{equation}
    (-\Delta) (-\Delta u) = (-\Delta)^2 u = Q \mathcal{H}^{n-1} \mres \Gamma 
\end{equation}
 and 
\begin{equation}
   (- \Delta) \left( \frac{Q}{2} |d_\Gamma| \right) = - Q \; \mathcal{H}^{n-1}\mres \Gamma - \mathrm{tr}(g),
\end{equation}
where $g= (g_{ij})$ is as in Lemma \ref{lem:distprep}. We conclude that (again in the sense of distributions in $C_0^\infty(B_\epsilon(\Gamma))'$) one has 
\begin{equation}
   (-\Delta) \left( \Delta u - \frac{Q}{2}|d_\Gamma| \right)  =  \mathrm{tr}(g)  \in L^p(B_\epsilon(\Gamma)). 
\end{equation}
Therefore, by elliptic regularity $h := \Delta u - \frac{Q}{2} |d_\Gamma|$ lies in $W^{2,p}_{loc}(B_\epsilon(\Gamma))$. Using this and Lemma \ref{lem:distprep} once more we infer from \eqref{eq:prfintermediate} 
\begin{align}
    \int \partial^2_{ij} u \Delta \phi \; \mathrm{d}x & = \int h \partial^2_{ij} \phi \; \mathrm{d}x  + \int \left( \frac{Q}{2} |d_\Gamma| \right) \partial^2_{ij}\phi \; \mathrm{d}x \\ & = \int (\partial^2_{ij} h) \phi \; \mathrm{d}x + \int_\Gamma Q \nu_i \nu_j \phi \; \mathrm{d}\mathcal{H}^{n-1} + \int g_{ij} \phi \; \mathrm{d}x.
\end{align}
We infer that (distributionally in $C_0^\infty(B_\epsilon(\Gamma))'$) one has 
\begin{equation}\label{eq:deltaDij}\tag{AUX1}
    \Delta(\partial^2_{ij} u) = \nu_i \nu_j Q \; \mathcal{H}^{n-1} \mres \Gamma + (\partial^2_{ij} h + g_{ij}). 
\end{equation}
 Notice that $\partial^2_{ij}h + g_{ij} \in L^p_{loc}(B_\epsilon(\Gamma))$. Therefore we can decompose the second derivative on $B_\epsilon(\Gamma)$ as a sum of two functions, namely $\partial^2_{ij}u = w_1+w_2$ where $w_1 \in L^1_{loc}(B_\epsilon(\Gamma))$ solves (distributionally in $C_0^\infty(B_\epsilon(\Gamma))'$)
\begin{equation}
    \Delta w_1 = \nu_i \nu_j Q \; \mathcal{H}^{n-1} \mres \Gamma.
\end{equation}
and $w_2 \in L^1_{loc}(B_\epsilon(\Gamma))$ solves  (distributionally in $C_0^\infty(B_\epsilon(\Gamma))'$)
\begin{equation}
    \Delta w_2 = (\partial^2_{ij} h + g_{ij}) \in L^p_{loc}(B_\epsilon(\Gamma)).
\end{equation}
We remark that this decomposition is not unique, only up to harmonic functions on $B_\epsilon(\Gamma)$. However, since these are smooth, they do not play a role for the (local) regularity. 
From Lemma \ref{lem:LocLip} we conclude that $w_1 \in W^{1,\infty}_{loc}(B_\epsilon(\Gamma))$ and elliptic regularity yields $w_2 \in W^{2,p}_{loc}(B_\epsilon(\Gamma))$. Therefore we obtain $\partial^2_{ij} u =w_1+w_2 \in W^{1,\infty}_{loc}(B_\epsilon(\Gamma))$. Since by Lemma \ref{lem:LocBiharm} $u \in C^\infty(\overline{\Omega \setminus B_{\frac{\epsilon}{2}}(\Gamma))})$ we also have $\partial^2_{ij} u \in C^\infty(\overline{\Omega \setminus B_{\frac{\epsilon}{2}}(\Gamma)})$. 
Since $i,j \in \{ 1,...,n\}$ were arbitrary we find that $u \in W^{3,\infty}(\Omega)$. 
\end{proof}

\section{Regularity inheritance (Proof of Theorem \ref{thm:bvreg})}

Next we address the question of regularity inheritance for our fourth order problem.
To this end we first need some notation. We say that a continuous function $f : \Gamma \rightarrow \mathbb{R}$ lies in $ C^{1,1}(\Gamma)$ if $f \in C^1(\Gamma)$ and there exists $L > 0$ such that 
\begin{equation}\label{eq:Liptan}
    |\nabla_\Gamma f(x) - \nabla_\Gamma f(y)| \leq L|x-y| \quad \forall x,y \in \Gamma.  
\end{equation}
Here all norms are taken in $\mathbb{R}^n$ and $\nabla_\Gamma f(x)$ is defined as in \cite[p. 42]{Simon} (and always looked at as a vector in $\mathbb{R}^n$ by means of the usual identification $T_x\Gamma \subset \mathbb{R}^n$). We now turn to an extension result which will be of use for the proof. 
\begin{lemma}\label{lem:besov}
Let $\Gamma = \partial \Omega' \in C^{2,1}$. Then each function in $f \in C^{1,1}(\Gamma)$ 
can be extended to a function in $W^{2,\infty}(\Omega)$. 
\end{lemma}
\begin{proof}
Let $f \in C^{1,1}(\Gamma)$ and $L>0$ be as in \eqref{eq:Liptan}. By \cite[Section 14.6]{GilTru} and the fact that $\partial \Omega' \in C^{2,1}$ we know that for some suitably small $\varepsilon_0 \in (0, \mathrm{dist}(\Gamma,\Omega^C))$ the function $\Phi: \Gamma \times (-\varepsilon_0,\varepsilon_0) \rightarrow \Omega$ given by $\Phi(x) := x + t \nu(x)$ is a $C^{1,1}$-regular diffeomorphism onto its image. In particular, for any fixed $\varepsilon \in (0,\varepsilon_0)$ the set $T :=  \Phi(\Gamma \times (-\varepsilon,\varepsilon))$ is an open neighborhood of $\Gamma$ with $C^{1,1}$-boundary. Due to this fact we have $C^{1,1}(\bar{T}) = W^{2,\infty}(T)$, cf. \cite[Chapter 5.8.2.(b)]{Evans}. One readily checks that $\Phi^{-1}: T \rightarrow \Gamma \times (-\varepsilon,\varepsilon)$ actually lies in $C^{1,1}(\overline{T}; \mathbb{R}^n)$, say $\Phi^{-1}$ has Lipschitz constant $M_0>0$ and $D\Phi^{-1}$ has Lipschitz constant $M_1> 0$ on $\overline{T}$. Define now the functions $\Psi: T \rightarrow \Gamma$ and $\eta: T \rightarrow (-\varepsilon,\varepsilon)$ via the relation $\Phi^{-1}(z) = (\Psi(z), \eta(z))$ for all $z \in T$ and observe that both $\Psi$ and $\eta$ enjoy $C^{1,1}$-regularity on $\bar{T}$. Further, we define an extension $\bar{f}: T \rightarrow \mathbb{R}$ of $f$ via $
\bar{f}(z ) := f(\Psi(z)).$
We show now that $\bar{f} \in C^{1,1}(\overline{T})$. To this end first notice that for $ i = 1,...,n$ and all $z \in T$ one has
\begin{equation}
\partial_i \bar{f}(z)  := \langle \nabla_\Gamma f(\Psi(z)), \partial_i \Psi(z)\rangle ,
\end{equation}
where $\langle \cdot,\cdot \rangle$ denotes the Euclidean scalar product in $\mathbb{R}^n$. Now for $z_1,z_2 \in T$ and $i = 1,...,n$  we have 
\begin{align}
 & |\partial_i \bar{f}(z_2) - \partial_i \bar{f}(z_1)|  = |\langle \nabla_\Gamma f(\Psi(z_2)), \partial_i \Psi(z_2)\rangle  - \langle \nabla_\Gamma f(\Psi(z_1)), \partial_i \Psi(z_1)\rangle | 
\\ & = |\langle \nabla_\Gamma f(\Psi(z_2)) - \nabla_\Gamma f(\Psi(z_1)), \partial_i \Psi(z_2)\rangle  - \langle \nabla_\Gamma f(\Psi(z_1)), \partial_i \Psi(z_1) - \partial_i \Psi(z_2) \rangle |
\\ & = |\nabla_\Gamma f(\Psi(z_2)) - \nabla_\Gamma f(\Psi(z_1))|  \; | \partial_i \Psi(z_2) |  + | \nabla_\Gamma f(\Psi(z_1))| \; | \partial_i \Psi(z_1) - \partial_i \Psi(z_2) |
\\ & \leq L |\Psi(z_1)-\Psi(z_2)| \; |\partial_i \Psi(z_2)| +  ||f||_{C^1(\Gamma)} M_1 |z_2-z_1|
 \leq (L M_0^2 + ||f||_{C^1(\Gamma)} M_1)|z_2-z_1|.
\end{align}
We conclude that $\bar{f} \in C^{1,1}(\bar{T}) = W^{2,\infty}(T)$. As $T \subset \Omega$ is an open neighborhood of $\Gamma$ one readily checks by a suitable cutoff argument that also an extension in $W^{2,\infty}(\Omega)$ can be found. 
\end{proof}

Making use of this and the already derived auxiliary equation \eqref{eq:deltaDij} for the second derviatives of solutions, we are able to prove Theorem \ref{thm:bvreg}. 

\begin{proof}[Proof of Theorem \ref{thm:bvreg}]
Assume $\Gamma = \partial \Omega' \in C^{2,1}$. 
Following the lines of the proof of Theorem \ref{thm:optireg}, we derive as in \eqref{eq:deltaDij} that distibutionally in $C^\infty(B_\epsilon(\Gamma))'$ there holds 
\begin{equation}
    \Delta (\partial^2_{ij} u) = \nu_i\nu_j Q \; \mathcal{H}^{n-1}\mres \Gamma + f_{ij}
\end{equation}
for some $f_{ij} \in L^p_{loc}(B_\epsilon(\Gamma))$ and all $1 \leq i,j \leq n$. In particular, one can decompose as in the proof of Theorem \ref{thm:optireg} $\partial^2_{ij} u = w_1 + w_2$ where $w_2 \in W^{2,p}_{loc}(B_\epsilon(\Gamma))$ solves (distributionally) $\Delta w_2 =  f_{ij}$ and $w_1 \in L^1_{loc}(B_\epsilon(\Gamma))$ solves (distributionally)
\begin{equation}
    \Delta w_1 = \nu_i \nu_j Q  \; \mathcal{H}^{n-1} \mres \Gamma.
\end{equation}
Since $\nu_i \nu_j \in C^{1,1}(\Gamma)$ one can find an extension in $W^{2,p}(\Omega)$. Since 
also $Q \in W^{2,p}(\Omega)$ and $W^{2,p}(\Omega)$ is a Banach algebra (as $p> n$), we infer that $\Delta w_1 = \tilde{Q}  \; \mathcal{H}^{n-1}\mres \Gamma$ for some $\tilde{Q} \in W^{2,p}(B_\epsilon(\Gamma))$. Thereupon, Lemma \ref{lem:alles} yields that $\nabla w_1 \in BV_{loc}(B_\epsilon(\Gamma))$. Moreover, one has $\nabla w_2 \in W^{1,p}_{loc}(B_\epsilon(\Gamma)) \subset BV_{loc}(B_\epsilon(\Gamma))$.  All in all we infer $ \nabla (\partial^2_{ij} u ) = \nabla w_1 + \nabla w_2 \in BV_{loc}(B_\epsilon(\Gamma))$. Hence, $\partial^3_{ijk} u \in BV_{loc}(B_\epsilon(\Gamma))$ for all $1 \leq i,j,k \leq n$. The fact that $u$ is smooth on $\Omega \setminus B_{\frac{\epsilon}{2}}(\Gamma)$ and also in a neighborhood of $\partial \Omega$ (cf. Lemma \ref{lem:LocBiharm}) finally implies that $D^3 u \in BV(\Omega)$. 
\end{proof}

\subsection{SBV-Regularity for $D^3u$}\label{sec:SBV}

The following $SBV$-regularity statement for the higher order equation follows immediately from the results in Section \ref{sec:modelprob} and the auxiliary problem that has been formed in \eqref{eq:deltaDij}.

\begin{cor}\label{cor:SBV}
Let $\Gamma = \partial \Omega' \in C^{2,1}$ and $Q \in W^{2,p}(\Omega)$, $p > n$. Further, let $u \in W^{2,2}(\Omega)$ be a weak solution of \eqref{eq:biharmeqmeas} with $\Gamma,Q$ as above. 
Then $u \in W^{3,\infty}(\Omega)$ and  $D^3 u \in SBV(\Omega)$. Moreover, $J_{\partial^3_{ijk} u} \subset  \Gamma$ for all $i,j,k = 1,...,n$.
\end{cor}
\begin{proof}
We conclude again following the lines of the proof of Theorem \ref{thm:optireg} (up to \eqref{eq:deltaDij}) that 
(distributionally in $C_0^\infty(B_\epsilon(\Gamma))'$) one has for all $i,j \in \{ 1,..., n\}$
\begin{equation}
    \Delta (\partial^2_{ij}u ) = \nu_i \nu_j Q \;  \mathcal{H}^{n-1} \mres \Gamma + f_{ij}
\end{equation}
for some $f_{ij} \in L^p_{loc}(B_\epsilon(\Gamma))$. Notice once again that there exists $w \in W^{2,p}_{loc}(B_\epsilon(\Gamma))$ auch that $\Delta w = f$. Therefore (distibutionally in $C_0^\infty(B_\epsilon(\Gamma))'$) one has
\begin{equation}
    \Delta (\partial^2_{ij} u - w ) = \nu_i \nu_j Q \; \mathcal{H}^{n-1} \mres \Gamma. 
\end{equation}
With (a slight adaptation of) Corollary \ref{cor:33} one concludes that $\nabla (\partial^2_{ij} u - w)  \in SBV_{loc}(B_\epsilon(\Gamma))$ and hence (since $\nabla w \in W^{1,p}_{loc}(B_\epsilon(\Gamma)) \subset SBV_{loc}(B_\epsilon(\Gamma))$) one has $\nabla \partial^2_{ij}u \in SBV_{loc}(B_\epsilon(\Gamma))$. Since also $\nabla\partial^2_{ij} u \in C^\infty( \overline{\Omega \setminus B_{\frac{\epsilon}{2}}(\Gamma)} )$ (cf. Lemma \ref{lem:LocBiharm}), we find that $\nabla \partial^2_{ij} u \in SBV(\Omega)$. The claim follows.
\end{proof}



 \section{Optimal regularity for the biharmonic Alt-Caffarelli problem}\label{sec:AltCaffarelli}
 The biharmonic Alt-Caffarelli problem deals with the minimization of 
 \begin{equation}\label{eq:probAlt}
     \mathcal{E}(u) = \int_\Omega (\Delta u)^2 \; \mathrm{d}x + |\{ u > 0\}|
 \end{equation}
 among all $u \in W^{2,2}(\Omega)$ such that $u - u_0 \in W_0^{1,2}(\Omega)$ for some $u_0 \in C^\infty(\overline{\Omega})$ such that $u_0 > 0$. Here $\Omega \subset \mathbb{R}^n$ is a smooth domain. It has raised some interest in recent literature, cf. \cite{Serena1}, \cite{Serena2}, \cite{AltCaffarelliMarius}. 
 Minimizers have to balance out two competing interests: On the one hand bending has to be minimized, but on the other hand minimizers must stay below the zero level on a large set. The set $\Gamma := \{ u  = 0 \}$ is a \emph{free boundary} of the problem, where minimizers lose regularity due to the fact that $\mathcal{E}$ is not Frechet  differentiable with respect to perturbations whose support intersects $\Gamma$.
 
 \begin{figure}[ht]
    \centering
    \begin{subfigure}[b]{0.9\textwidth}
        \includegraphics[width=\textwidth]{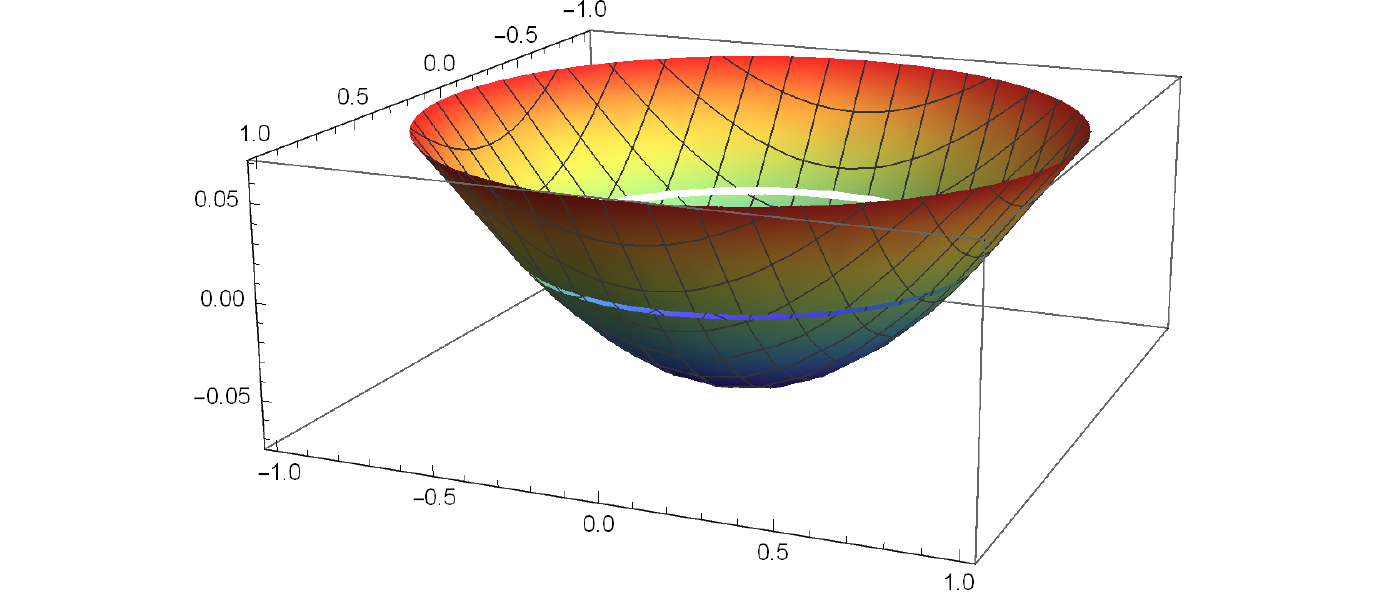}
        \caption{3D-Plot of a minimizer $u$}
    \end{subfigure}
    
    \begin{subfigure}[b]{0.48\textwidth}
        \includegraphics[width=\textwidth]{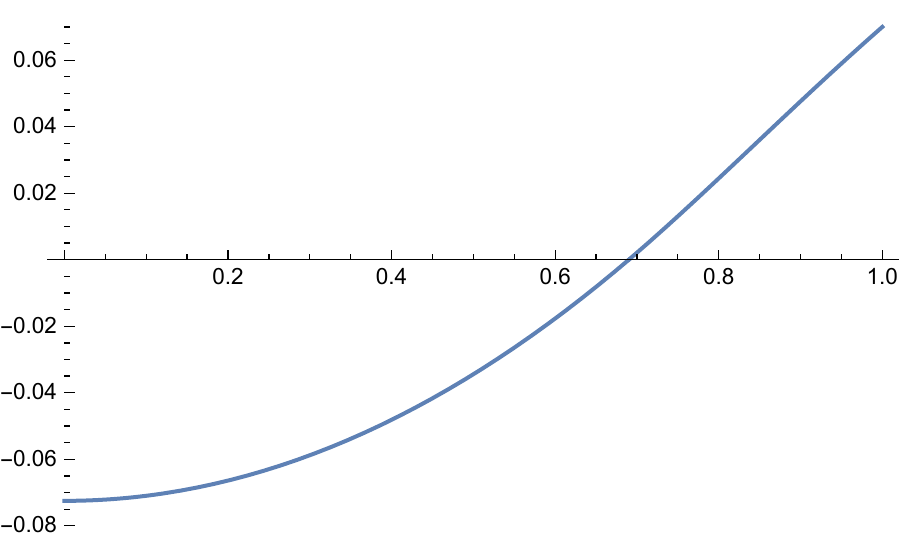}
        \caption{Radial profile curve $r \mapsto u(re_1)$}
    \end{subfigure}  
    ~
    \begin{subfigure}[b]{0.48\textwidth}
        \includegraphics[width=\textwidth]{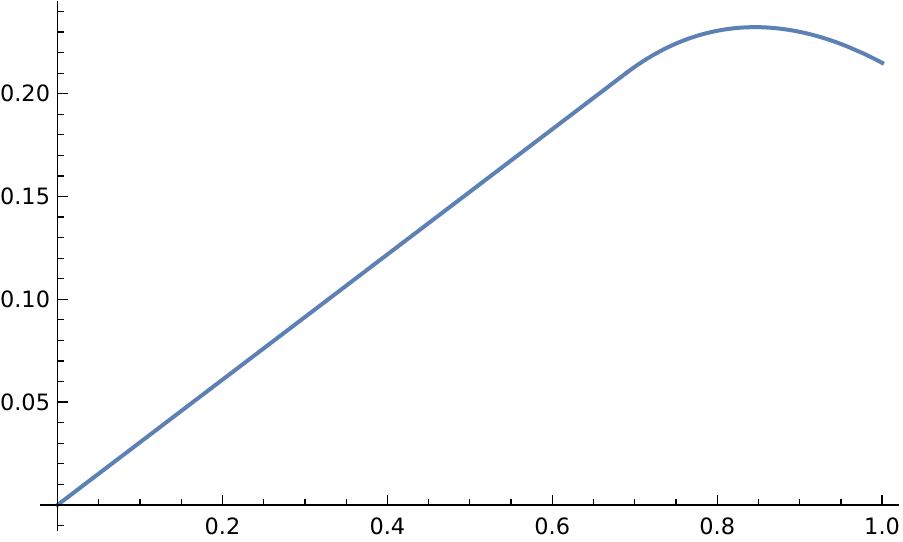}
        \caption{First radial derivative $r \mapsto \partial_r u(re_1)$}
    \end{subfigure}
    
    \begin{subfigure}[b]{0.48\textwidth}
        \includegraphics[width=\textwidth]{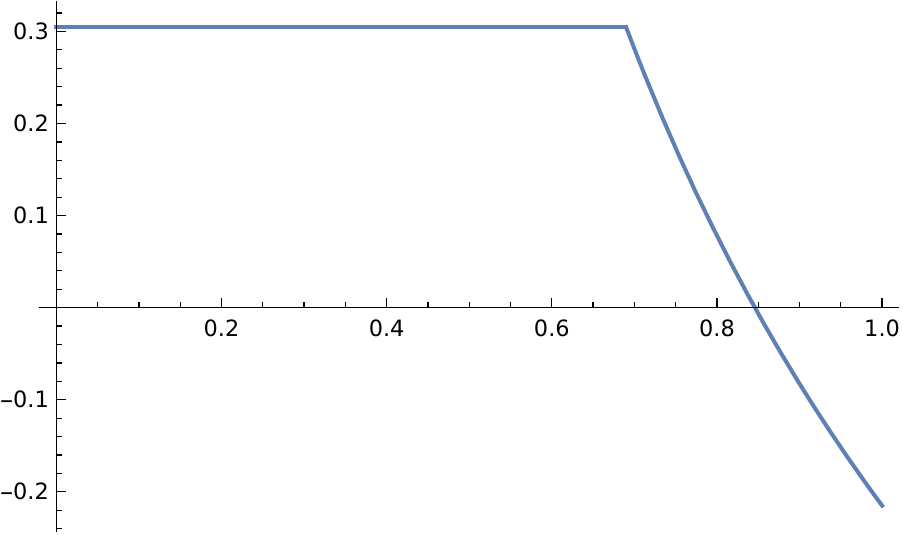}
        \caption{Second radial derivative $r \mapsto \partial_r^2 u(re_1)$}
    \end{subfigure}
    ~
    \begin{subfigure}[b]{0.48\textwidth}
        \includegraphics[width=\textwidth]{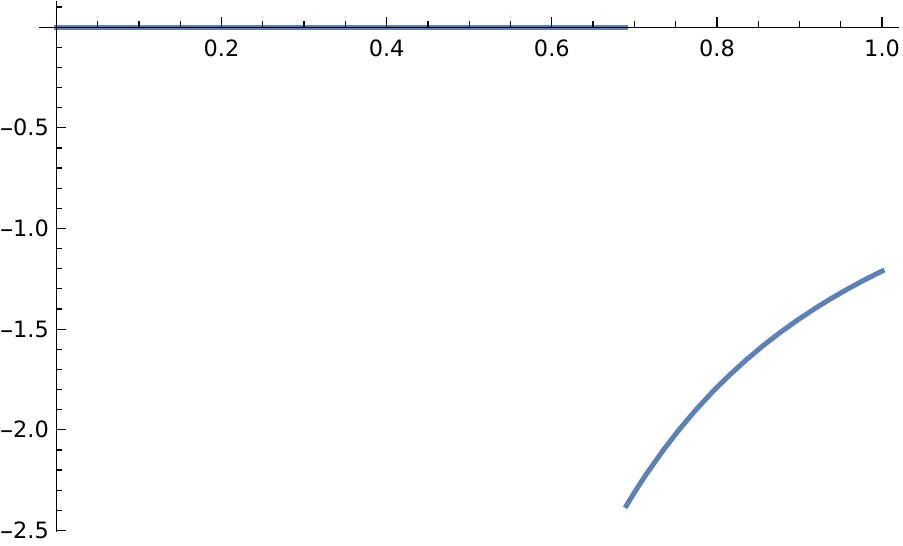}
        \caption{Third radial derivative $r \mapsto \partial_r^3 u(re_1)$}
    \end{subfigure}  
    \caption{A minimizer $u$ for $ \Omega=B_1(0) \subset \mathbb{R}^2$ with constant boundary datum $u_0 \equiv 0.07.$ In this situation, radial symmetry of  minimizers and explicit formulas are proved in \cite[Section 10]{AltCaffarelliMarius}. }\label{fig:OrbitlikeFreeElastica}
\end{figure}
 
 The behavior of minimizers in dimension $n = 2$ is studied in \cite{AltCaffarelliMarius}. 
 
 \begin{theorem}[{\cite[Theorem 1.4]{AltCaffarelliMarius}}]\label{thm:AltCaffa}
 Suppose that $n = 2$ and $\Omega \subset \mathbb{R}^2$ is a smooth domain and $u_0 \in C^\infty(\overline{\Omega}), u_0 > 0$. Then each minimizer of the problem \eqref{eq:probAlt} lies in $C^2(\Omega)$ and satisfies $\nabla u \neq 0$ on the free boundary $\Gamma = \{ u = 0 \}$. Moreover, $\Gamma = \bigcup_{i = 1}^N \partial G_i$ for finitely many disjoint domains $G_i$ with $C^2$-smooth boundary (whose boundaries are also disjoint). Furthermore, the minimizer satisfies the equation 
 \begin{equation}\label{eq:ELeq}
     \int_\Omega \Delta u \Delta \phi \; \mathrm{d}x  = - \frac{1}{2} \int_\Gamma \frac{1}{|\nabla u |} \phi \; \mathrm{d}x \quad \forall \phi \in W^{2,2}(\Omega) \cap W_0^{1,2}(\Omega). 
 \end{equation}
 \end{theorem}
 
 
 In the language of this article, each minimizer $u$ is a weak solution of the equation 
\begin{equation}\label{eq:Alt-Caffarelli}
    \begin{cases}
    \Delta^2 u = Q \;  \mathcal{H}^{n-1} \mres \Gamma  & \mathrm{in } \; \Omega \\ 
    u = u_0 , \;  \Delta u = 0 & \mathrm{on }  \; \partial \Omega
    \end{cases}
\end{equation}
 where $Q = Q(u) = -\frac{1}{2|\nabla u|} \big\vert_{\Gamma}  \in C^{0,\alpha}(\Gamma)$.

By the discussion before \eqref{eq:Poiprob} and Remark \ref{rem:LaxMilgram} one already obtains that $v= -\Delta u \in W^{1,q}_0(\Omega)$ for all $q <\infty$ [whereupon elliptic regularity yields $u \in W^{3,q}(\Omega)$ for all $q \in [1,\infty)$].   
With Theorem \ref{thm:optireg} we are  finally able to establish the optimal $W^{3,\infty}$-regularity. 

\begin{theorem}[Optimal regularity for the  biharmonic Alt-Caffarelli problem]
Let $u \in W^{2,2}(\Omega)$ be a minimizer for the problem \eqref{eq:probAlt} in dimension $n = 2$. Then $u \in W^{3,\infty}(\Omega)$ and $u \not \in C^3(\Omega)$. Furthermore, $D^3u \in SBV(\Omega)$. 
\end{theorem}
\begin{proof}
We already know that $u \in W^{3,q}(\Omega)$ for all $q \in [1,\infty)$ and that $\nabla u \neq 0$ on $\Gamma = \{ u = 0 \}$. Notice in particular that $\nabla u \in W^{2,q}(\Omega)$ for all $q \in [1,\infty)$, and hence also $\nabla u \in C^{1,\alpha}(\overline{\Omega})$ for all $\alpha \in [0,1)$. 
 Thus there exists $\epsilon> 0$ such that $\nabla u \neq 0$ on $B_\epsilon ( \Gamma)$. Choose $\psi \in C^2_0(B_\epsilon(\Gamma))$ such that $\psi \equiv 1$ on $\Gamma$. 
We infer that $\tilde{Q} := \tilde{Q}(u) := - \frac{1}{2|\nabla u|} \psi$ lies in $W^{2,q}(\Omega)$ for all $q \in [1,\infty)$ (in particular for some $q> n$) and is an extension of $Q(u)$ to $\Omega$. In particular, \eqref{eq:Alt-Caffarelli} is equivalent to
\begin{equation}\label{eq:Alt-Caffa-schlange}
    \begin{cases}
    \Delta^2 u = \tilde{Q} \;  \mathcal{H}^{n-1} \mres \Gamma  & \mathrm{in } \; \Omega \\ 
    u = u_0 , \;  \Delta u = 0 & \mathrm{on }  \; \partial \Omega.
    \end{cases}
\end{equation}
If now $\Gamma = \partial G$ for a single $C^2$-domain $G$, Theorem \ref{thm:optireg} implies that \eqref{eq:Alt-Caffa-schlange} has a solution $u \in W^{3,\infty}(\Omega).$ In the general case of multiple boundary components $\Gamma = \bigcup_{i = 1}^N (\partial G_i)$  (disjoint union) one has to write $\mathcal{H}^{n-1} \mres \Gamma = \sum_{i = 1}^N \mathcal{H}^{n-1} \mres (\partial G_i)$ and use the linearity of the equation to infer the asserted $W^{3,\infty}$-regularity. Since by Theorem \ref{thm:AltCaffa} there are only finitely many connected components of $\Gamma$, the $W^{3,\infty}$-regularity is shown. 
That $u \not \in C^3(\Omega)$ follows from Lemma \ref{lem:alles}, applied to $v= -\Delta u$.

To derive the $BV$-regularity of $D^3 u$ we intend to apply Theorem \ref{thm:bvreg}. In order to do so, we need to show that $\Gamma = \{ u = 0 \}$ is a $C^{2,1}$-regular manifold. This is due to the fact that $\Gamma$ is the zero level set of the function $u \in W^{3,\infty}(\Omega) = C^{2,1}(\Omega)$ and $\nabla u \neq 0$ on $\Gamma$. Therefore the situation \eqref{eq:Alt-Caffa-schlange} satisfies all requirements of Theorem \ref{thm:bvreg} and we finally obtain that $D^3 u \in BV(\Omega)$. Furthermore $SBV$-regularity readily follows from \eqref{eq:Alt-Caffa-schlange} and Corollary \ref{cor:SBV}.
\end{proof}

\section{Generalization to the polyharmonic case}
The findings in the above article can be generalized to polyharmonic operators $(-\Delta)^m$, $m \geq 2$ on smooth domains $\Omega \subset \mathbb{R}^n$. For this we need some more notation. A \emph{boundary operator} $B(x,D): C^\infty(\overline{\Omega}) \rightarrow L^2(\partial \Omega)$ of order $\leq N$ is a formal expression of the form 
\begin{equation}
    B(x,D) = \sum_{ \alpha \in \mathbb{N}_0^n, |\alpha| \leq N}  b_\alpha(x) D^\alpha \quad \textrm{where $b_\alpha \in C^\infty(\overline{\Omega})$ for each multiindex $\alpha$.} 
\end{equation}
Notice that Sobolev trace theory implies that each boundary operator of order $\leq N$ extends to a continuous map $ B(x,D) : W^{N+1,2}(\Omega) \rightarrow L^2(\partial \Omega)$.

\begin{definition}\label{def:polyharm}
Let $m \geq 2$, $Q \in W^{2,p}(\Omega)$, $p> n$ and $\Gamma = \partial \Omega' \in C^2$. Let $B_j(x,D): W^{2m-3,2}(\Omega) \rightarrow L^2( \partial \Omega)$, $j = 1,...,m-1$ be boundary operators of order $N_j \leq 2m-4$ such that with $B_m(x,D) := (-\Delta)^{m-1}$ the set $(B_j(x,D))_{j = 1}^{m}$ are \emph{complementing boundary conditions} in the sense of \cite[Definition 2.9]{GGS}. A function $u \in W^{2m-2,2}(\Omega)$ is said to be a weak solution to
\begin{equation}
\begin{cases}
   \; \; \;    (-\Delta)^m u = Q \; \mathcal{H}^{n-1} \mres \Gamma & \textrm{in }\Omega, \\  \; B_j(x,D)u= B_j(x,D)u_0 , \; \; j = 1,...,m-1 & \textrm{on }\partial \Omega, \\ B_m(x,D)u = 0 & \textrm{on }\partial \Omega, \end{cases}
\end{equation}
if $B_j(x,D)(u-u_0) = 0$ for all $j = 1,...,m-1$ and 
\begin{equation}\label{eq:solipoli}
    \int_\Omega [(-\Delta)^{m-1} u] \Delta \phi \; \mathrm{d}x = \int_\Gamma Q \phi \; \mathrm{d}\mathcal{H}^{n-1} \quad \forall \phi \in C^2(\overline{\Omega}) \cap W_0^{1,2}(\Omega). 
\end{equation}
\end{definition}
The regularity theorem we can show here is 
\begin{theorem}
Let $u \in W^{2m-2,2}(\Omega)$ be as in Definition \ref{def:polyharm}. Then $u \in W^{2m-1,\infty}(\Omega)$. 
\end{theorem}
\begin{proof}
We proceed by induction over $m$. For $m= 2$ the statement is Theorem \ref{thm:optireg}. Next let $m > 2$ and $i,j \in \{ 1,...,m \}$. Let $\epsilon>0$ be chosen as in Section \ref{sec:prelom}. 
For $\phi \in C_0^\infty(B_\epsilon(\Gamma))$ we can compute 
\begin{equation}
    \int [(-\Delta)^{m-2} (\partial^2_{ij} u)] \Delta \phi \; \mathrm{d}x = \int [(-\Delta)^{m-2} u] \Delta (\partial^2_{ij} \phi) \; \mathrm{d}x = \int [(-\Delta)^{m-1} u] \partial^2_{ij} \phi \; \mathrm{d}x. 
\end{equation}
Defining $h_m :=[(-\Delta)^{m-1} u] +\frac{Q}{2}|d_\Gamma|$ where $d_\Gamma$ is as in Section \ref{sec:prelom} we obtain 
\begin{equation}
    \int [(-\Delta)^{m-2} (\partial^2_{ij} u)] \Delta \phi \; \mathrm{d}x = \int h_m \partial^2_{ij} \phi \; \mathrm{d}x - \int \frac{Q}{2} |d_\Gamma| \partial^2_{ij}\phi \; \mathrm{d}x.
\end{equation}
Using Lemma \ref{lem:distprep} we find that 
\begin{equation}\label{eq:poliaux}
    \int [(-\Delta)^{m-2} (\partial^2_{ij} u)] \Delta \phi \; \mathrm{d}x = \int h_m \partial^2_{ij} \phi \; \mathrm{d}x - \int_\Gamma Q \nu_i \nu_j \phi \;   \mathrm{d} \mathcal{H}^{n-1} - \int g_{ij} \phi \; \mathrm{d}x.
\end{equation}
for some $g_{ij} \in L^p(B_\epsilon(\Gamma))$. Next we show that $h_m \in W^{2,p}_{loc}(B_\epsilon(\Gamma))$. To this end observe that for each $\psi \in C_0^\infty(B_\epsilon(\Gamma))$ one has (by \eqref{eq:solipoli} and Lemma \ref{lem:distprep})  
\begin{align}
    \int h_m \Delta \psi \; \mathrm{d}x &  = \int_\Omega [(-\Delta)^{m-1} u] \Delta \psi \; \mathrm{d}x  + \int_\Omega \frac{Q}{2}|d_\Gamma| \Delta \psi \; \mathrm{d}x 
    \\ & = - \int_\Gamma Q \psi \; \mathrm{d}\mathcal{H}^{n-1} + \int_\Gamma Q \psi \; \mathrm{d} \mathcal{H}^{n-1}+ \int (\mathrm{tr}(g)) \psi \; \mathrm{d}x  = \int (\mathrm{tr}(g)) \psi \; \mathrm{d}x, 
\end{align}
where the matrix $g = (g_{ij}) \in L^p(B_\epsilon(\Gamma))$ is as above. Elliptic regularity yields then that $h_m \in W^{2,p}_{loc}(B_\epsilon(\Gamma))$ and \eqref{eq:poliaux} can be reformulated to 
\begin{equation}
    \int (-\Delta)^{m-2} (\partial^2_{ij}u ) \Delta \phi \; \mathrm{d}x = -\int_\Gamma Q \nu_i \nu_j \phi \; \mathrm{d}\mathcal{H}^{n-1} + \int (\partial^2_{ij} h_m - g_{ij}) \phi \; \mathrm{d}x.
\end{equation}
Distibutionally in  $C_0^\infty(B_\epsilon(\Gamma))'$ there holds therefore 
\begin{equation}\label{eq:auxiploi} \tag{AUX2}
    (-\Delta)^{m-1} (\partial^2_{ij}u)   = Q \nu_i \nu_j \; \mathcal{H}^{n-1} \mres \Gamma - (\partial^2_{ij} h_m - g_{ij}).
\end{equation}
Using the induction assumption and the fact that $\partial^2_{ij} h_m - g_{ij} \in L^p(B_\epsilon(\Gamma))$ one can proceed as in the discussion after \eqref{eq:deltaDij} to prove that $\partial^2_{ij} u \in W^{2 (m-1) - 1, \infty}_{loc}(B_\epsilon(\Gamma))$, i.e. $D^2 u \in W^{2m-3,\infty}_{loc}(B_\epsilon(\Gamma))$, implying $u \in W^{2m-1,\infty}_{loc}(B_\epsilon(\Gamma))$. Similar to Lemma \ref{lem:LocBiharm} one can also deduce that $u \in C^\infty(\overline{\Omega \setminus B_{\frac{\epsilon}{2}}(\Gamma)} )$ (this time using \cite[Theorem 2.20]{GGS}), yielding that $u \in W^{2m-1,\infty}(\Omega)$. 
\end{proof}

Keeping using the auxiliary equation \eqref{eq:auxiploi}, one can also deduce generalizations of the $BV$-regularity result (Theorem \ref{thm:bvreg}) and also of the $SBV$-regularity result  in Section \ref{sec:SBV}. All in all one obtains that for interfaces $\Gamma = \partial \Omega' \in C^{2,1}$ one has $D^{2m-1} u \in SBV(\Omega)$. The details can be safely omitted as they follow the lines of the previous arguments (with the usage of the auxiliary equation \eqref{eq:auxiploi} instead of \eqref{eq:deltaDij} and the usage of elliptic regularity for $(-\Delta)^{m-1}$ instead of $(-\Delta)^1$).








\end{document}